\documentclass[11pt]{amsart}
\usepackage{amssymb}
\usepackage{amsmath}
\usepackage{amsfonts}
\usepackage{graphicx}

\usepackage[total={17cm,22cm},top=2.5cm, left=2.3cm]{geometry}
\usepackage{hyperref}
    \usepackage{aeguill}
    \usepackage{type1cm}

\theoremstyle{plain}
\newtheorem{thm}{Theorem}[section]

\newtheorem{definition}[thm]{Definition}
\newtheorem{example}[thm]{Example}

\newtheorem{lemma}[thm]{Lemma}

\newtheorem{proposition}[thm]{Proposition}
\newtheorem{remark}[thm]{Remark}

\newtheorem{theorem}[thm]{Theorem}
\numberwithin{equation}{section}
\newcommand{\N}{\mathbb{N}}

\newcommand{\R}{\mathbb{R}}

\newcommand{\Rn}{\mathbb{R}^n}

\DeclareMathOperator{\lip}{Lip\,\!}

\usepackage[usenames,dvipsnames]{color}

\usepackage[dvipsnames]{xcolor}

\begin{document}

\title[Explicit formulas for $C^{1,1}$ extensions of $1$-jets]{Explicit formulas for $C^{1, 1}$ and $C^{1, \omega}_{\textrm{conv}}$ extensions of $1$-jets in Hilbert and superreflexive spaces}
\author{D. Azagra}
\address{ICMAT (CSIC-UAM-UC3-UCM), Departamento de An{\'a}lisis Matem{\'a}tico,
Facultad Ciencias Matem{\'a}ticas, Universidad Complutense, 28040, Madrid, Spain }
\email{azagra@mat.ucm.es}

\author{E. Le Gruyer}
\address{INSA de Rennes \& IRMAR, 20, Avenue des Buttes de Co\"esmes, CS 70839
F-35708, Rennes Cedex 7, France}
\email{Erwan.Le-Gruyer@insa-rennes.fr}

\author{C. Mudarra}
\address{ICMAT (CSIC-UAM-UC3-UCM), Calle Nicol\'as Cabrera 13-15.
28049 Madrid, Spain}
\email{carlos.mudarra@icmat.es}

\date{May 26, 2017}


\thanks{D. Azagra was partially supported by Ministerio de Educaci\'on, Cultura y Deporte, Programa Estatal de Promoci\'on del Talento y su Empleabilidad en I+D+i, Subprograma Estatal de Movilidad. C. Mudarra was supported by Programa Internacional de Doctorado Fundaci\'on La Caixa--Severo Ochoa. Both authors partially suported by grant MTM2015-65825-P}


\begin{abstract}
Given $X$ a Hilbert space, $\omega$ a modulus of continuity, $E$ an arbitrary subset of $X$, and functions $f:E\to\R$, $G:E\to X$, we provide necessary and sufficient conditions for the jet $(f,G)$ to admit an extension $(F, \nabla F)$ with $F:X\to\R$ convex and of class $C^{1, \omega}(X)$, by means of a simple explicit formula. As a consequence of this result, if $\omega$ is linear, we show that a variant of this formula provides explicit $C^{1,1}$ extensions of general (not necessarily {\em convex}) $1$-jets satisfying the usual Whitney extension condition, with best possible Lipschitz constants of the gradients of the extensions. Finally, if $X$ is a superreflexive Banach space, we establish similar results for the classes $C^{1, \alpha}_{\textrm{conv}}(X)$.
\end{abstract}

\maketitle

\section{Introduction and main results}

If $C$ is a subset of $\R^n$ and we are given functions $f:C\to\R$, $G:C\to\R^n$, the $C^{1,1}$ version of the classical Whitney extension theorem (see \cite{Whitney, Glaeser, Stein} for instance) theorem tells us that there exists a function $F\in C^{1,1}(\R^n)$ with $F=f$ on $C$ and $\nabla F=G$ on $C$ if and only if the $1$-jet $(f,G)$ satisfies the following property: there exists a constant $M>0$ such that
$$
|f(x)-f(y)-\langle G(y), x-y\rangle|\leq M |x-y|^2, \,\,\, \textrm{ and } \,\,\,
|G(x)-G(y)|\leq M|x-y| \eqno(\widetilde{W^{1,1}})
$$
for all $x, y\in C.$ We can trivially extend $(f,G)$ to the closure $\overline{C}$ of $C$ so that the inequalities $( \widetilde{W^{1,1}})$ hold on $\overline{C}$ with the same constant $M.$ The function $F$ can be explicitly defined by
$$
F(x)=\begin{cases}
f(x) & \mbox{ if } x\in \overline{C} \\
\sum_{Q\in\mathcal{Q}}\left( f(x_Q)+\langle G(x_Q), x-x_Q\rangle\right)\varphi_{Q}(x)  & \mbox{ if }  x\in\R^n\setminus \overline{C},
\end{cases}
$$
where $\mathcal{Q}$ is a family of {\em Whitney cubes} that cover the complement of the closure $\overline{C}$ of $C, \{\varphi_{Q}\}_{Q\in\mathcal{Q}}$ is the usual Whitney partition of unity associated to $\mathcal{Q}$, and $x_Q$ is a point of $\overline{C}$ which minimizes the distance of $\overline{C}$ to the cube $Q.$ Recall also that the function $F$ constructed in this way has the property that $\lip(\nabla F) \leq k(n)M,$ where $k(n)$ is a constant depending only on $n$ (but going to infinity as $n\to\infty$), and $\lip(\nabla F)$ denotes the Lipschitz constant of the gradient $\nabla F$.

In \cite{Wells, LeGruyer1} it was shown, by very different means, that this $C^{1,1}$ version of the Whitney extension theorem holds true if we replace $\R^n$ with any Hilbert space and, moreover, there is an extension operator $(f,G)\mapsto (F, \nabla F)$ which is minimal, in the following sense. Given a Hilbert space $X$  with norm denoted by $\|\cdot\|$, a subset $E$ of $X$, and functions $f:E\to\R$, $G:E\to X$, a necessary and sufficient condition for the $1$-jet $(f,G)$ to have a $C^{1,1}$ extension $(F, \nabla F)$ to the whole space $X$ is that
\begin{equation}\label{LeGruyersCondition}
\Gamma(f,G,E):=\sup_{x,y\in E} \left( \sqrt{A_{x,y}^2+B_{x,y}^2} + |A_{x,y}| \right) < \infty,
\end{equation}
where
$$
A_{x,y} = \frac{2(f(x)-f(y))+\langle G(x)+G(y),y-x \rangle}{\| x-y\|^2} \quad \text{and}
$$
$$
 B_{x,y} = \frac{\|G(x)-G(y)\|}{\|x-y\|} \quad \text{for all} \quad x,y\in E, x\neq y.
$$
Moreover, the extension $(F, \nabla F)$ can be taken with best Lipschitz constants, in the sense that
$$
\Gamma(F, \nabla F, X) = \Gamma(f,G,E) = \|(f,G)\|_{E},
$$
where
$$
\|(f,G)\|_{E}:= \inf \lbrace \lip(\nabla H) \: : \: H \in C^{1,1}(X) \:\: \text{and} \:\: (H,\nabla H) = (f,G) \:\: \text{on} \:\: E \rbrace
$$
is the trace seminorm of the jet $(f,G)$ on $E$; see \cite{LeGruyer1} and \cite[Lemma 15]{LeGruyer2}.

While the operators $(f,G)\mapsto (F, \nabla F)$ given by the constructions in \cite{Wells, LeGruyer1, LeGruyer2} are not linear, they have the useful property that, when we put them to work on $X=\R^n$, they satisfy $\lip(\nabla F)\leq\eta \|(f,G)\|_{E}$ for some $\eta>0$ independent of $n$ (in fact for $\eta=1$); hence one can say that they are bounded, with norms independent of the dimension $n$, provided that we endow $C^{1,1}(X)$ with the seminorm given by $\lip (\nabla F)$ and equip the space of jets $(f,G)$ with the  trace seminorm $\|(f,G)\|_E$. In contrast, the Whitney extension operator is linear and also bounded in this sense, but with norm going to $\infty$ as $n\to\infty$).
On the negative side, the formulas in \cite{LeGruyer2} depending on Wells's construction are more complicated than the proof of \cite{LeGruyer1}, which uses Zorn's lemma and in particular is not constructive. 
For more information about Whitney extension problems and extension operators see \cite{BrudnyiShvartsman, Fefferman2005, Fefferman2006, FeffermanSurvey, FeffermanIsraelLuli1, FeffermanIsraelLuli2, Glaeser, VossHirnMcCollum, JSSG, LeGruyer2, DacorognaGangbo, Valentine} and the references therein.

In this paper, among other things, we will remedy those two drawbacks by providing a very simple, explicit formula for $C^{1,1}$ extension of jets in Hilbert spaces: let us say that a jet $(f,G)$ on $E\subset X$ satisfies condition $(W^{1,1})$ provided that there exists a number $M>0$ such that
$$
f(y) \leq f(x) + \frac{1}{2} \langle G(x)+G(y), y-x \rangle + \frac{M}{4} \|x-y\|^2 - \frac{1}{4M} \|G(x)-G(y)\|^2 \eqno(W^{1,1})
$$
for all $x,y\in E$. This condition is equal to Wells's necessary and sufficient condition in \cite[Theorem 2]{Wells}. Also, it is easy to check that this condition is absolutely equivalent to $(\widetilde{W^{1,1}}),$ meaning that if $(W^{1,1})$ is satisfied with some constant $M>0,$ then $(\widetilde{W^{1,1}})$ is satisfied with constant $ k M,$ (where $k$ is an absolute constant independent of the space $X$; in particular $k$ does not depend on the dimension of $X$), and viceversa. Moreover, $(W^{1,1})$ is also absolutely equivalent to \eqref{LeGruyersCondition} and, in fact, the number $\Gamma(f,G,E)$ is the smallest $M>0$ for which $(f,G)$ satisfies $(W^{1,1})$ with constant $M>0$; see \cite[Lemma 15]{LeGruyer2}.

In Theorem \ref{theoremformulaC11nonnecessaryconvex} below we will show that, for every $(f,G)$ defined on $E$ and satisfying the property $(W^{1,1})$ with constant $M$ on $E,$ the formula
\begin{align}\label{formula for C11 extension}
& F=\textrm{conv}(g)- \tfrac{M}{2}\| \cdot \|^2, \, \textrm{ where } \\
& g(x) = \inf_{y \in E} \lbrace f(y)+\langle G(y), x-y \rangle + \tfrac{M}{2} \|x-y\|^2 \rbrace + \tfrac{M}{2}\|x\|^2 , \quad x\in X,
\end{align}
defines a $C^{1,1}(X)$ function with $F=f$ and $\nabla F =G$ on $E$ and $\lip(\nabla F) \leq M.$ 
Here $\textrm{conv}(g)$ denotes the convex envelope of $g$, defined by
$$
\textrm{conv}(g)(x)=\sup\{ h(x) \, : \, h \textrm{ is convex, proper and lower semicontinuous }, h\leq g\}.
$$
Another expression for $\textrm{conv}(g)$ is given by
$$
\textrm{conv}(g)(x)=\inf\left\lbrace \sum_{j=1}^{k}\lambda_{j} g(x_j) \, : \, \lambda_j\geq 0,
\sum_{j=1}^{k}\lambda_j =1, \, x=\sum_{j=1}^{k}\lambda_j x_j, \, k\in\N \right\rbrace.
$$
In the case that $X$ is finite dimensional, say $X=\R^n$, this expression can be made simpler: by using Carath\'eodory's Theorem one can show that it is enough to consider convex combinations of at most $n+1$ points. That is to say, if $g:\R^n\to\R$ then
$$
\textrm{conv}(g)(x)=\inf\left\lbrace \sum_{j=1}^{n+1}\lambda_{j} g(x_j) \, : \, \lambda_j\geq 0,
\sum_{j=1}^{n+1}\lambda_j =1, \, x=\sum_{j=1}^{n+1}\lambda_j x_j \right\rbrace;
$$
see \cite[Corollary 17.1.5]{Rockafellar} for instance.

Let us informally explain the reasons why formula \eqref{formula for C11 extension} does its job. It is well known that a function $F:X \to\R$ is of class $C^{1,1}$, with $\textrm{Lip}(\nabla F)=M$, if and only if $F+\frac{M}{2}\|\cdot\|^2$ is convex and $F-\frac{M}{2}\|\cdot\|^2$ is concave. So, if we are given a $1$-jet $(f,G)$ defined on $E\subset X$ which can be extended to $(F, \nabla F)$ with $F\in C^{1,1}(X)$ and $\textrm{Lip}(\nabla F)\leq M$, then the function $H=F+\frac{M}{2}\|\cdot\|^2$ will be convex and of class $C^{1,1}$. Conversely, if we can find a convex and $C^{1,1}$ function $H$ such that $(H,\nabla H)$ is an extension of the jet $E \ni y \mapsto \left( f(y)+\frac{M}{2}\|y\|^2, G(y) + M y \right),$ then $X \ni y \mapsto \left( H(y)-\frac{M}{2}\|y\|^2, \nabla H(y)-M y \right)$ will be a $C^{1,1}$ extension of $(f,G).$
Thus we can reduce the $C^{1,1}$ extension problem for jets to the $C^{1,1}_{\textrm{conv}}$ extension problem for jets. Here, as in the rest of the paper, $C^{1,1}_{\textrm{conv}}(X)$ will stand for the set of all convex functions $\varphi:X\to\R$ of class $C^{1,1}$.

Now, how can we solve the $C^{1,1}_{\textrm{conv}}$ extension problem for jets? In \cite{AM} the following necessary and sufficient condition for $C^{1,1}_{\textrm{conv}}$ extension of jets was given: for any $E\subset \R^n$, $f: E \to \R, \: G : E \to X,$ we say that $(f,G)$ satisfies condition $(CW^{1,1})$ on $E$ with constant $M>0,$ provided that
$$
f(x) \geq f(y)+ \langle G(y), x-y \rangle + \frac{1}{2M} \| G(x)-G(y)\|^2 \quad \text{for all} \quad x,y\in E. 
$$
In \cite{AM} it is shown that a jet $(f, G)$ has an extension $(F, \nabla F)$ with $F\in C^{1,1}_{\textrm{conv}}$ if and only if $(f,G)$ satisfies $(CW^{1,1})$; moreover in this case one can take $F\in C^{1,1}_{\textrm{conv}}$ such that $\textrm{Lip}(\nabla F)\leq k(n)M$, where $k(n)$ is a constant only depending on $n$. The construction in \cite{AM} is explicit, but has the same disadvantage as the Whitney extension operator has, namely that $\lim_{n\to\infty}k(n)=\infty$. In \cite{AMHilbert} this result is extended to the Hilbert space setting, but the proof, inspired by \cite{LeGruyer1}, is not constructive. However, by following the ideas of the proof of \cite{AM}, but using a simple formula instead of the Whitney extension theorem, we will show in Theorem  \ref{theoremformulaC11convex} below that if a $1$-jet $(f,G)$ defined on a subset $E$ of a Hilbert space satisfies condition $(CW^{1,1})$ then the function $F$ defined by
$$
F=\textrm{conv}(g), \quad \textrm{where} \quad g(x) = \inf_{y \in E} \lbrace f(y)+\langle G(y), x-y \rangle + \tfrac{M}{2} \|x-y\|^2 \rbrace, \quad x\in X,
$$
is a $C^{1,1}$ convex function such that $F_{|_E}=f$, $(\nabla F)_{|_E}=G$, and $\lip(\nabla F) \leq M$. Moreover, if $H$ is another $C^{1,1}$ convex function with $H=f$ and $\nabla H=G$ on $E$ and $\lip(\nabla H) \leq M,$ then $H \leq F$. This strategy allows us to solve the $C^{1,1}_{\textrm{conv}}$ extension problem for jets with best constants and, after checking that if $(f, G)$ satisfies $(W^{1,1})$ then $\left(f(y)+\frac{M}{2}\|y\|^2, G(y)+M y \right)$ satisfies $(CW^{1,1})$, also allows us to show that the expression
$$
F(x)=\textrm{conv}\left( z\mapsto\inf_{y\in E}\{f(y)+\tfrac{M}{2}\|y\|^2+\langle G(y)+My, z-y\rangle +M\|z-y\|^2\}\right) (x) -\tfrac{M}{2}\|x\|^2,
$$
which is easily seen to be equal to 
\eqref{formula for C11 extension}, provides an extension formula that solves the minimal $C^{1,1}$ extension problem for jets, in the sense that $\textrm{Lip}(\nabla F)\leq M$. Besides we will also prove that if $H$ is another $C^{1,1}$ function with $H=f$ and $\nabla H=G$ on $E$ and $\lip(\nabla H) \leq M$, then $H \leq F$. Since the extension of $(f,G)$ constructed by Wells in \cite{Wells} also has this property, it follows that in fact \eqref{formula for C11 extension} coincides with Wells's extension. The point is of course that both our formula \eqref{formula for C11 extension} and the proof that it works are much simpler than Wells's construction and proof.

Moreover, the latent potential in this kind of formula, at least in the convex case, is not confined to $C^{1,1}$ extension problems in Hilbert spaces. Indeed, on the one hand, by means of a similar formula, we will show in Theorem \ref{theoremformulac1omegaconvex} below that, if $X$ is a Hilbert space and $\omega$ is a concave, strictly increasing, modulus of continuity, with $\omega(\infty)=\infty$, then the condition $(CW^{1, \omega})$ of \cite{AM} is necessary and sufficient for a $1$-jet $(f,G)$ defined on a subset $E$ of a Hilbert space to have an extension $(F, \nabla F)$ such that $F:X\to\R$ is convex and of class $C^{1, \omega}$, with 
$$
\sup_{x,y\in X, \, x\neq y}\frac{\|\nabla F(x)-\nabla F(y)\|}{\omega(\|x-y\|)}  \leq 8 M.
$$
Not only does this provide a new result\footnote{Of course, Theorem \ref{theoremformulac1omegaconvex} is essentially much more general than Theorem \ref{theoremformulaC11convex}, but we deliberately present these two results in two different sections of this paper, for the following two reasons. 1) In Theorem \ref{theoremformulaC11convex} we are able to obtain best possible Lipschitz constants of the gradients of the extension, whereas in Theorem \ref{theoremformulac1omegaconvex} we only get them up to a factor 8. 2) The proof of Theorem \ref{theoremformulac1omegaconvex} is more technical and uses some machinery from Convex Analysis, such as Fenchel conjugates, smoothness and convexity moduli, etc, which could obscure the main ideas and prevent some readers interested only in the proofs of the $C^{1,1}$ results from easily understanding them.} for the infinite-dimensional case, but also shows that the constants $k$ can be supposed to be independent of the dimension $n$ in \cite
 [Theorem 1.4]{AM}, at least if $\omega(\infty)=\infty$ (and in particular for all of the classes $C^{1, \alpha}_{\textrm{conv}}(\R^n)$). On the other hand,  we will see in Section 5 that one can even go beyond the Hilbertian case and show that a similar result holds for the class $C^{1, \alpha}_{\textrm{conv}}(X)$ whenever $(X, \|\cdot\|)$ is a superreflexive Banach space whose norm $\|\cdot\|$ has modulus of smoothness of power type $1+\alpha,$ with $\alpha\in (0,1]$; this is the content of Theorem \ref{theoremformulac1alphaconvex} below. Finally, in Section 6 we give an example showing that all of the above results fail in the Banach space $c_0$.

Unfortunately, it seems very unlikely that one could use this kind of formulas to solve $C^{1, \alpha}$ extension problems for general (not necessarily convex) $1$-jets in superreflexive\footnote{It is well known that superreflexive Banach spaces are characterized as being Banach spaces with equivalent norms of class $C^{1, \alpha}$ for some $\alpha\in (0, 1]$, and Hilbert spaces are characterized as being Banach spaces with equivalent norms of class $C^{1,1}$. For general reference about renorming properties of superreflexive spaces see, for instance \cite{DGZ, FabianEtAl}.} Banach spaces with $C^{1, \alpha}$ equivalent norms if $\alpha\neq 1$. The exponent $\alpha=1$ is somewhat miraculous in this respect: even for the simplest case that $X=\R$, it is not true in general that, given a function $f\in C^{1, \alpha}(\R)$, there exists a constant $C$ such that $f+C|\cdot|^{1+\alpha}$ is convex.  

\medskip

When the first version of this paper was completed, a preprint of A. Daniilidis, M. Haddou, E. Le Gruyer and O. Ley \cite{DaniilidisHaddouLeGruyerLey} concerning the same problem in Hilbert spaces was made public. The formula for $C^{1,1}_{\textrm{conv}}$ extension of $1$-jets given by \cite{DaniilidisHaddouLeGruyerLey} is different from the formula we provide in this paper. As these authors show, their formula cannot work for the H\"older differentiability classes $C^{1, \alpha}_{\textrm{conv}}$ with $\alpha\neq 1$. Two advantages of the present approach are the fact that our formula does work for theses classes, and its simplicity.

\section{Optimal $C^{1,1}$ convex extensions of $1$-jets by explicit formulas in Hilbert spaces}\label{sectionc11convex}

Given an arbitrary subset $E$ of $X,$ and a $1$-jet $f: E \to \R, \: G : E \to X,$ we will say that $(f,G)$ satisfies the condition $(CW^{1,1})$ on $E$ with constant $M>0,$ provided that
$$
f(x) \geq f(y)+ \langle G(y), x-y \rangle + \frac{1}{2M} \| G(x)-G(y)\|^2 \quad \text{for all} \quad x,y\in E.
$$
The following Proposition shows that this condition is necessary for a $1$-jet to have a $C^{1,1}$ convex extension to all of $X$.
\begin{proposition}\label{necessityconditioncw11}
Let $f\in C^{1,1}(X)$ be convex, and assume that $f$ is not affine. Then
$$
f(x)-f(y)-\langle \nabla f(y), x-y\rangle \geq \frac{1}{2M} \|\nabla f(x)-\nabla f(y)\|^2
$$
for all $x, y\in X$, where 
$$
M=\sup_{x, y\in X, \, x\neq y}\frac{\|\nabla f(x)-\nabla f(y)\|}{\|x-y\|}.
$$
\end{proposition}
\noindent On the other hand, if $f$ is affine, it is obvious that $(f, \nabla f)$ satisfies $(CW^{1,1})$ on every $E\subset X$, for every $M>0$. 

For a proof of the above Proposition, see \cite[Proposition 2.1]{AMHilbert}, or Proposition \ref{necessityconditioncw11} below in a more general form.

We will need to use the following characterization of $C^{1,1}$ differentiability of convex functions. Of course the result is well known, but we will provide a short proof for completeness, and also in order to remark that the implication $(ii)\implies (i)$ is true for not necessarily convex functions as well, a fact that we will have to use later on.

\begin{proposition}\label{characterizationdifferentiability}
For a continuous convex function $f: X \to \R,$ the following statements are equivalent.
\item[$(i)$] There exists $M>0$ such that
$$
f(x+h)+f(x-h)-2f(x) \leq M \|h\|^2 \quad \text{for all} \quad x,h \in X.
$$
\item[$(ii)$] $f$ is differentiable on $X$ with $\lip( \nabla f ) \leq M.$
\end{proposition}
\begin{proof} 
First we prove that $(ii)$ implies $(i)$, which is also valid for non-convex functions. Using that $\lip( \nabla f ) \leq M,$ it follows from Taylor's theorem that
$$
f(x+h)-f(x)- \langle \nabla f (x), h \rangle  \leq \frac{M}{2} \|h\|^2.
$$
Similarly we have 
$$
f(x-h)-f(x)- \langle \nabla f (x), -h \rangle \leq \frac{M}{2} \|h\|^2,
$$
and combining both inequalities we get $(i)$. Now we do assume that $f$ is a convex function and let us show that $(i)$ implies $(ii)$. Since
$$
\lim_{h \to 0} \frac{f(x+h)+f(x-h)-2f(x)}{\|h\|} = 0,
$$
for all $x\in X$ and $f$ is convex and continuous, $f$ is differentiable on $X$. In order to  prove that $\lip( \nabla f ) \leq M$ it is enough to see that the function $F: X \to \R$ defined by $F(x)=\frac{M}{2}\|x\|^2 - f (x), \: x\in X,$ is convex. Since $f$ is a continuous function, the convexity of $F$ is equivalent to: 
$$
F \left( \frac{x+y}{2} \right) \leq \tfrac{1}{2} F(x)+ \tfrac{1}{2} F(y) \quad \text{for all} \quad x,y \in X.
$$
To see this, given $x,y \in X,$ we can write
\begin{align*}
F \left( \frac{x+y}{2} \right) = \tfrac{1}{2} F(x)+ \tfrac{1}{2} F(y) + \frac{1}{2} \left( f(x)+f(y)-2 f \left( \frac{x+y}{2} \right) - M \Big \| \frac{x-y}{2} \Big \|^2 \right)
\end{align*}
Applying $(ii)$ with $h = \frac{x-y}{2}$ we obtain that 
$$
f(x)+f(y)-2 f \left( \frac{x+y}{2} \right) \leq  M \Big \| \frac{x-y}{2} \Big \|^2,
$$
which in turns implies $ F \left( \frac{x+y}{2} \right) \leq \tfrac{1}{2} F(x)+ \tfrac{1}{2} F(y).$ 
\end{proof}

Recall that for a function $f:X \to\R,$ the convex envelope of $f$ is defined by
$$
\textrm{conv}(f)(x)=\sup\{ \phi(x) \, : \, \phi \textrm{ is convex and lsc }, \phi\leq f\}.
$$
Another expression for $\textrm{conv}(f)$ is:
$$
\textrm{conv}(f)(x)=\inf \left\lbrace \sum_{j=1}^{n}\lambda_{j} f(x_j) \, : \, \lambda_j\geq 0,
\sum_{j=1}^{n}\lambda_j =1, x=\sum_{j=1}^{n}\lambda_j x_j, n\in\N \right\rbrace,
$$

The following result shows that the operator $f\mapsto \textrm{conv}(f)$ not only preserves $C^{1,1}$ smoothness of functions $f$ and Lipschitz constants of their gradients $\nabla f$, but also that, even for some nondifferentiable functions $f$, their convex envelopes $\textrm{conv}(f)$ will be of class $C^{1,1}$, with best possible constants, provided that the functions $f$ satisfy suitable one-sided estimates. This is a slight (but very significant for our purposes) improvement of particular cases of the results in \cite{GriewankRabier}, \cite[Theorem 7]{CepedelloRegularization}, and \cite{KirchheimKristensen}.

\begin{theorem}\label{convexenvelopeconstant}
Let $X$ be a Banach space. Suppose that a function $f: X \to \R$ has a convex, lower semicontinous minorant, and satisfies
$$
f(x+h)+f(x-h)-2f(x) \leq M \|h\|^2 \quad \text{for all} \quad x,h \in X.
$$
Then $\psi:= \textrm{conv}(f)$ is a continuous convex function satisfying the same property. In view of Proposition \ref{characterizationdifferentiability}, we conclude that $\psi$ is of class $C^{1,1}( X)$, with $\lip ( \nabla \psi ) \leq M.$ In particular, for a function $f \in C^{1,1}(X),$ we have that $\textrm{conv}(f) \in C^{1,1}(X)$, with $\lip( \nabla \psi) \leq \lip ( \nabla f).$
\end{theorem}
\begin{proof}
The function $\psi$ is well defined as $\psi \leq f$ and $f$ has a convex, lsc minorant. Now let us check the mentioned inequality. Given $x, h \in X$ and $\varepsilon>0,$ we can pick $n\in \N, \: x_1, \ldots, x_n \in X$ and $\lambda_1, \ldots, \lambda_n >0$ such that 
$$
\psi (x) \geq \sum_{i=1}^n \lambda_i f(x_i) - \varepsilon, \quad \sum_{i=1}^n \lambda_i=1 \quad \text{and} \quad \sum_{i=1}^n \lambda_i x_i = x.
$$
Since $x\pm h = \sum_{i=1}^n \lambda_i ( x_i \pm h ),$ we have $\psi( x \pm h ) \leq \sum_{i=1}^n \lambda_i f( x_i \pm h).$ This leads us to
$$
\psi(x+h)+\psi(x-h)-2\psi(x) \leq \sum_{i=1}^n \lambda_i \left( f(x_i+h)+f(x_i-h)-2f(x_i) \right) + 2 \varepsilon
$$
By the assumption on $f,$ we obtain
$$
 f(x_i+h)+f(x_i-h)-2f(x_i) \leq M \|h\|^2 \quad i=1, \ldots,n.
 $$
Therefore
\begin{equation}\label{inequalityconvexenvelope}
\psi(x+h)+\psi(x-h)-2\psi(x) \leq M \|h\|^2 + 2\varepsilon.
\end{equation}
Since $\varepsilon>0$ is arbitrary, we get the desired inequality. It is clear that $\psi,$ being a supremum of a family of lower semicontinuous convex functions that are pointwise uniformly bounded (by the function $f$), is convex, proper and lower semicontinuous.
And because all lower semicontinuous, proper and convex functions are continuous at interior points of their domains (see \cite[Proposition 4.1.5]{BorweinVanderwerffbook} for instance), we also have that $\psi$ is continuous. 
\end{proof}

\begin{theorem}\label{theoremformulaC11convex}
Given a $1$-jet $(f,G)$ defined on $E$ satisfying property $(CW^{1,1})$ with constant $M$ on $E,$ the formula
$$
F=\textrm{conv}(g), \quad g(x) = \inf_{y \in E} \lbrace f(y)+\langle G(y), x-y \rangle + \tfrac{M}{2} \|x-y\|^2 \rbrace, \quad x\in X,
$$
defines a $C^{1,1}$ convex function such that $F_{|_E}=f$, $(\nabla F)_{|_E}=G$, and $\lip(\nabla F) \leq M$. 

Moreover, if $H$ is another $C^{1,1}$ convex function with $H_{|_E}=f$, $(\nabla H)_{|_E}=G$, and $\lip(\nabla H) \leq M$, then $H \leq F$.
\end{theorem}
\begin{proof}
The proof follows the lines of that of \cite[Theorem 1.4]{AM}, but will be considerably simplified by applying Theorem \ref{convexenvelopeconstant} to the function $g$ defined in the statement (instead of applying the result from \cite{KirchheimKristensen} to a function arising from a more elaborate construction involving Whitney's classical extension techniques with dyadic cubes and associated partitions of unity). It is worth noting that the function $g$ is not differentiable in general. Nonetheless $F=\textrm{conv}(g)$ is of class $C^{1,1}$ because, as we next show, $g$ satisfies the one-sided estimate of Theorem \ref{convexenvelopeconstant}.

\begin{lemma}\label{estimationinfimumc11}
We have
$$
g(x+h)+g(x-h)-2g(x) \leq M \| h \|^2 \quad \text{for all} \quad x,h \in X.
$$
\end{lemma}
\begin{proof}
Given $x,h \in X$ and $\varepsilon>0,$ by definition of $g,$ we can pick $y\in E$ with
$$
g(x) \geq f(y)+ \langle G(y), x-y \rangle + \tfrac{M}{2} \| x-y\|^2 - \varepsilon.
$$
We then have
\begin{align*}
g(x+h)& +g(x-h)-2g(x)  \leq f(y)+\langle G(y), x+h-y \rangle + \tfrac{M}{2} \|x+h-y\|^2 \\
& \quad + f(y)+\langle G(y), x-h-y \rangle + \tfrac{M}{2} \|x-h-y\|^2 \\
& \quad -2 \left( f(y)+ \langle G(y), x-y \rangle +  \tfrac{M}{2} \|x-y\|^2 \right) + 2 \varepsilon \\
& = \tfrac{M}{2} \left( \|x+h-y\|^2 + \|x-h-y\|^2- 2 \| x-y\|^2 \right) + 2 \varepsilon \\
& = M \| h \|^2 +2 \varepsilon.
\end{align*}
Since $\varepsilon$ is arbitrary, the above proves our Lemma.

\end{proof}

\begin{lemma}\label{minimalsmallerthaninf}
We have that
$$
f(z)+ \langle G(z), x-z \rangle \leq f(y)+\langle G(y), x-y \rangle + \tfrac{M}{2} \|x-y\|^2 
$$
for every $y,z\in E, \: x\in X.$ 
\end{lemma}
\begin{proof}
Given $y,z\in E, \: x\in X,$ condition $(CW^{1,1})$ implies
\begin{align*}
 & f(y) +\langle G(y), x-y \rangle + \tfrac{M}{2} \|x-y\|^2   \\
 & \geq f(z)+ \langle G(z), y-z \rangle + \tfrac{1}{2M}\| G(y)-G(z)\|^2 + \langle G(y), x-y \rangle + \tfrac{M}{2} \|x-y\|^2 \\
 & = f(z)+ \langle G(z),x-z \rangle + \tfrac{1}{2M}\| G(y)-G(z)\|^2 + \langle G(z)-G(y), y-x \rangle + \tfrac{M}{2} \|x-y\|^2 \\
 & = f(z)+ \langle G(z),x-z \rangle + \tfrac{1}{2M} \| G(y)-G(z) + 2M( y-x) \|^2 \\
 & \geq f(z)+ \langle G(z),x-z \rangle. 
\end{align*}
\end{proof}

The preceding lemma shows that $m \leq g,$ where $g$ is defined as in Theorem \ref{theoremformulaC11convex}, and 
$$
m(x) := \sup_{z \in E} \lbrace f(z)+ \langle G(z), x-z \rangle \rbrace, \quad x\in X.
$$
Bearing in mind the definitions of $g$ and $m$ we then deduce that $f \leq m \leq g \leq f$ on $E.$ Thus $g = f$ on $E.$ We also note that the function $m$, being a supremum of continuous functions, is lower semicontinuous on $X.$ By Lemma \ref{estimationinfimumc11} and Theorem  \ref{convexenvelopeconstant} we then obtain that $F= \textrm{conv}(g)$ is convex and of class $C^{1,1}$, with $\lip(\nabla F) \leq M.$ Since $m$ is convex, by definition of $F,$ we have $m \leq F \leq g,$ where both $m$ and $g$ coincide with $f$ on $E.$ Thus $F=f$ on $E$.

Also, note that $m \leq F$ on $X$ and $F=m$ on $E,$ where $m$ is convex and $F$ is differentiable on $X.$ This implies that $m$ is differentiable on $E$ with $\nabla m (x)= \nabla F (x) $ for all $x\in E$. It is clear, by definition of $m,$ that $G(x) \in \partial m (x)$ (denoting the subdifferential of $m$ at $x$) for every $x\in E$, and this observation shows that $\nabla F = G$ on $E.$ 

Finally, consider another convex extension $H \in C^{1,1}(X)$ of the jet $(f,G)$ with $\lip(\nabla H) \leq M.$ Using Taylor's theorem and the assumptions on $H$ we have that
$$
H (x)  \leq f(y)+ \langle  G(y), x-y \rangle + \tfrac{M}{2} \| x-y\|^2, \quad x\in X, \:y \in E.
$$

Taking the infimum over $y\in E$ we get $H \leq g$ on $X.$ On the other hand, bearing in mind that $H$ is convex, the definition of the convex envelope of a function implies
$H=\textrm{conv}(H) \leq \textrm{conv}(g)=F$ on $X$. 
This completes the proof of Theorem \ref{theoremformulaC11convex}.
\end{proof}

\section{Optimal $C^{1,1}$ extensions of $1$-jets by explicit formulas in Hilbert spaces}

In this section we will prove that formula \eqref{formula for C11 extension} defines a $C^{1,1}$ extension of the jet $(f, G)$ on $E$, provided that this jet satisfies a necessary and sufficient condition found by Wells in \cite{Wells}, which is equivalent to the classical Whitney condition for $C^{1,1}$ extension $(\widetilde{W^{1,1}})$.

\begin{definition}
We will say that a $1$-jet $(f,G)$ defined on a subset $E$ of a Hilbert space satisfies condition $(W^{1,1})$ with constant $M>0$ on $E$ provided that
$$
f(y) \leq f(x) + \frac{1}{2} \langle G(x)+G(y), y-x \rangle + \frac{M}{4} \|x-y\|^2 - \frac{1}{4M} \|G(x)-G(y)\|^2
$$
for all $x,y\in E$.
\end{definition}

Let us first see why this condition is necessary.
\begin{proposition}\label{necessityconditionw11}
\item[$(i)$] If $(f,G)$ satisfies $(W^{1,1})$ on $E$ with constant $M,$ then $G$ is $M$-Lipschitz on $E.$
\item[$(ii)$] If $F$ is a function of class $C^{1,1}(X)$ with $\lip(\nabla F) \leq M,$ then $(F, \nabla F)$ satisfies $(W^{1,1})$ on $E=X$ with constant $M.$
\end{proposition}
\begin{proof}
\item[$(i)$] Given $x,y\in E,$ we have
\begin{align*}
f(y) \leq f(x) + \frac{1}{2} \langle G(x)+ G(y), y-x \rangle + \frac{M}{4} \|x-y\|^2 - \frac{1}{4M} \| G(x)- G (y)\|^2\\
f(x) \leq f(y) + \frac{1}{2} \langle G(y)+G(x), x-y \rangle + \frac{M}{4} \|x-y\|^2 - \frac{1}{4M} \|G(x)-G(y)\|^2.
\end{align*}
By combining both inequalities we easily get $ \| G(x)-G(y)\| \leq M \| x-y\|.$ 
\item[$(ii)$] Fix $x,y\in X$ and $z= \frac{1}{2}(x+y)+ \frac{1}{2M}( \nabla F(y)- \nabla F(x)).$ Using Taylor's theorem we obtain
$$
F(z)  \leq F(x) + \langle \nabla F(x), \tfrac{1}{2}(y-x) \rangle  + \tfrac{1}{2M}( \nabla F(y)-\nabla F(x)) \rangle  
 + \tfrac{M}{2}\big \| \tfrac{1}{2}(y-x) + \tfrac{1}{2M}( \nabla F(y)-\nabla F(x)) \big \|^2 
$$
and
$$
F(z)  \geq F(y) + \langle \nabla F(y), \tfrac{1}{2}(x-y) \rangle + \tfrac{1}{2M}( \nabla F(y)- \nabla F(x)) \rangle 
 - \tfrac{M}{2}\big \| \tfrac{1}{2}(x-y) + \tfrac{1}{2M}( \nabla F(y)-\nabla F(x)) \big \|^2 .
$$
Then we easily get
\begin{align*}
F(y) & \leq F(x) + \langle \nabla F(x), \tfrac{1}{2}(y-x) \rangle  + \tfrac{1}{2M}( \nabla F(y)-\nabla F(x)) \rangle \\
& \quad  + \tfrac{M}{2}\big \| \tfrac{1}{2}(y-x) \rangle + \tfrac{1}{2M}( \nabla F(y)-\nabla F(x)) \big \|^2 \\
& \quad -\langle \nabla F(y), \tfrac{1}{2}(x-y) \rangle - \tfrac{1}{2M}( \nabla F(y)-\nabla F(x)) \rangle \\
& \quad + \tfrac{M}{2}\big \| \tfrac{1}{2}(x-y) + \tfrac{1}{2M}( \nabla F(y)- \nabla F(x)) \big \|^2 \\
& = F(x) + \frac{1}{2} \langle \nabla F(x)+ \nabla F(y), y-x \rangle + \tfrac{M}{4} \|x-y\|^2 - \tfrac{1}{4M} \|\nabla F(x)-\nabla F(y)\|^2
\end{align*}

\end{proof}

The following lemma will allow us to deal with the $C^{1,1}$ extension problem for $1$-jets by relying on our previous solution of the $C^{1,1}$ convex extension problem for $1$-jets. 

\begin{lemma}\label{lemmafromconvextononconvex}
Given an arbitrary subset $E$ of a Hilbert space $X$ and a $1$-jet $(f,G)$ defined on $E$, we have the following: $(f,G)$ satisfies $(W^{1,1})$ on $E$, with constant $M>0$, if and only if the $1$-jet $(\tilde{f}, \tilde{G})$ defined by
$
\tilde{f}(x)=f(x) + \frac{M}{2} \|x\|^2, \: \tilde{G}(x)=G(x) + M x , \: x\in E,
$
satisfies property $(CW^{1,1})$ on $E$, with constant $2 M$. 
\end{lemma}

\begin{proof}
Suppose first that  $(f,G)$ satisfies $(W^{1,1})$ on $E$ with constant $M>0.$ We have, for all $x,y\in E,$
\begin{align*}
&  \tilde{f}(x)-\tilde{f}(y)- \langle \tilde{G}(y),x-y \rangle- \frac{1}{4M} \| \tilde{G}(x)-\tilde{G}(y)\|^2  \\
& = f(x) -f(y)+ \frac{M}{2} \|x\|^2-\frac{M}{2} \|y\|^2- \langle G(y)+ My, x-y \rangle \\
& \quad - \frac{1}{4M} \| G(x)-G(y)+ M(x-y) \|^2 \\
& \geq \frac{1}{2} \langle G(x)+G(y), x-y \rangle - \frac{M}{4} \|x-y\|^2 + \frac{1}{4M} \|G(x)-G(y)\|^2 \\
& \quad + f(x) -f(y)+ \frac{M}{2} \|x\|^2-\frac{M}{2} \|y\|^2- \langle G(y)+ My, x-y \rangle \\
& \quad - \frac{1}{4M} \| G(x)-G(y)+ M(x-y) \|^2 \\
& = \frac{M}{2} \|x\|^2+\frac{M}{2} \|y\|^2 - M \langle x,y \rangle- \frac{M}{2}\| x-y\|^2=0.
\end{align*}
Conversely, if $(\tilde{f}, \tilde{G})$ satisfies $(CW^{1,1})$ on $E$ with constant $2M,$ we have
\begin{align*}
& f(x)+ \frac{1}{2} \langle G(x)+G(y), y-x \rangle + \frac{M}{4} \| x-y\|^2 - \frac{1}{4M} \| G(x)-G(y)\|^2-f(y) \\
& = \tilde{f}(x)- \frac{M}{2}\| x\|^2 + \frac{1}{2} \langle \tilde{G}(x)+ \tilde{G}(y) - M(x+y), y-x \rangle + \frac{M}{4} \| x-y\|^2 \\
& \quad - \frac{1}{4M}\| \tilde{G}(x)-\tilde{G}(y)-M(x-y)\|^2 - \tilde{f}(y)+ \frac{M}{2}\| y \|^2 \\
& = \tilde{f}(x)-\tilde{f}(y) + \frac{1}{2} \langle \tilde{G}(x)+ \tilde{G}(y), y-x \rangle + \frac{M}{4} \| x-y\|^2 \\
& \quad - \frac{1}{4M} \| \tilde{G}(x)-\tilde{G}(y)-M(x-y)\|^2 \\
& \geq \langle \tilde{G}(y),x-y \rangle + \frac{1}{4M} \| \tilde{G}(x)-\tilde{G}(y)\|^2 + \frac{1}{2} \langle \tilde{G}(x)+\tilde{G(y)},y-x \rangle \\
& \quad + \frac{M}{4} \| x-y \|^2 - \frac{1}{4M} \| \tilde{G}(x)-\tilde{G}(y)-M(x-y) \|^2 = 0.
\end{align*}

\end{proof}

\begin{theorem}\label{theoremformulaC11nonnecessaryconvex}
Let $E$ be a subset of a Hilbert space $X$. Given a $1$-jet $(f,G)$ satisfying property $(W^{1,1})$ with constant $M$ on $E$, the formula
\begin{align*}
& F=\textrm{conv}(g)- \tfrac{M}{2}\| \cdot \|^2,\\
& g(x) = \inf_{y \in E} \lbrace f(y)+\langle G(y), x-y \rangle + \tfrac{M}{2} \|x-y\|^2 \rbrace + \tfrac{M}{2}\|x\|^2 , \quad x\in X,
\end{align*}
defines a $C^{1,1}(X)$ function with $F_{|_E}=f$, $(\nabla F)_{|_E} =G$, and $\lip(\nabla F) \leq M$. 

Moreover, if $H$ is another $C^{1,1}$ function with $H=f$ and $\nabla H=G$ on $E$ and $\lip(\nabla H) \leq M,$ then $H \leq F.$
\end{theorem}
\begin{proof}
From Lemma \ref{lemmafromconvextononconvex}, we know that the jet $(\tilde{f}, \tilde{G})$ defined by
$$
\tilde{f}(x)=f(x) + \frac{M}{2} \|x\|^2, \quad \tilde{G}(x)=G(x) + M x , \quad x\in E,
$$
satisfies property $(CW^{1,1})$ on $E$ with constant $2 M.$ Then, by Theorem \ref{theoremformulaC11convex}, the function
$$
\tilde{F} = \textrm{conv} (g), \quad \tilde{g}(x) = \inf_{y \in E} \lbrace \tilde{f}(y)+\langle \tilde{G}(y), x-y \rangle + M \|x-y\|^2 \rbrace, \quad x\in X,
$$
is convex and of class $C^{1,1}$ with $(\tilde{F}, \nabla \tilde{F})=( \tilde{f}, \tilde{G})$ on $E$ and $\lip(\nabla \tilde{F} ) \leq 2 M$. By an easy calculation we get that
$$
\tilde{g}(x) = \inf_{y \in E} \lbrace f(y)+\langle G(y), x-y \rangle + \tfrac{M}{2} \|x-y\|^2 \rbrace + \tfrac{M}{2}\|x\|^2, \quad x\in X.
$$
Now, according to Proposition \ref{necessityconditioncw11}, the jet $( \tilde{F}, \nabla \tilde{F})$ satisfies condition $(CW^{1,1})$  with constant $2M$ on the whole $X.$ Thus, if $F$ is the function defined by
$$
F(x)= \tilde{F}(x)- \frac{M}{2}\|x\|^2,\quad x\in X,
$$
we get, thanks to Lemma \ref{lemmafromconvextononconvex}, that the jet $(F, \nabla F)$ satisfies condition $(W^{1,1})$ with constant $M$ on $X$. Hence, by Proposition \ref{necessityconditionw11}, $F$ is of class $C^{1,1}(X)$, with $\lip(\nabla F) \leq M$. From the definition of $ \tilde{f}, \tilde{G}, \tilde{F}$ and $F$ it is immediate that $F=f$ and $\nabla F = G$ on $E$. 

Finally, suppose that $H$ is another $C^{1,1}(X)$ function with $H=f$ and $\nabla H=G$ on $E$ and $\lip(\nabla H) \leq M.$ Using all of these assumptions together with Taylor's Theorem we have that
$$
H(x) + \frac{M}{2}\| x\|^2 \leq f(y)+ \langle G(y), x-y \rangle + \frac{M}{2} \| x-y\|^2 + \frac{M}{2}\|x \|^2,
$$
for all $x\in X, y \in E.$ Taking the infimum over $E$ we get that 
$$
H(x) + \frac{M}{2}\| x \|^2 \leq g(x), \quad x\in X.
$$
Since $H$ is $C^{1,1}(X)$ with $\lip(\nabla H) \leq M,$ the jet $(H, \nabla H)$ satisfies the condition $(W^{1,1})$ on $E$ with constant $M.$ Using Lemma \ref{lemmafromconvextononconvex}, we obtain that $(\tilde{H}, \nabla \tilde{H})$ (defined as in that Lemma) satisfies $(CW^{1,1})$ on $E$ with constant $2M.$ In particular the function $X \ni x \mapsto \tilde{H}(x)=H(x) + \frac{M}{2} \| x \|^2$ is convex, which implies that
$$
 \tilde{H} = \textrm{conv}(\tilde{H}) \leq g.
 $$
Therefore, $\tilde{H} \leq \tilde{F}$ on $X$, from which we obtain that $H\leq F$ on $X$.
\end{proof}

\section{$C^{1,\omega}$ convex extensions of $1$-jets by explicit formulas in Hilbert spaces}

Throughout this section, unless otherwise stated, we will assume that $X$ is a Hilbert space and $\omega: [0,+\infty) \to [0,+ \infty)$ is a concave and increasing function such that $\omega(0)=0$ and $\lim_{t \to +\infty} \omega(t)=+\infty.$ Also, we will denote
\begin{equation}\label{definition of varphi}
\varphi(t)= \int_0^t \omega(s) ds
\end{equation}
for every $t \geq 0$. It is obvious that $\varphi$ is differentiable with $\varphi'=\omega$ on $[0,+\infty)$ and, because $\omega$ is strictly increasing, $\varphi$ is strictly convex. The function $\omega$ has an inverse $\omega^{-1}: [0,+ \infty) \to [0,+\infty)$ which is convex and  strictly increasing, with $\omega^{-1}(0)=0.$ We also note that
\begin{align*}
\omega( c t ) \leq c \omega(t) \quad \text{and} \quad \omega^{-1}( ct ) \geq c \omega^{-1}(t) \quad \text{for} \quad c \geq 1, \: t \geq 0 \\
\omega( c t ) \geq c \omega(t) \quad \text{and} \quad \omega^{-1}( ct ) \leq c \omega^{-1}(t) \quad \text{for} \quad c \leq 1, \: t \geq 0.
\end{align*}

In the sequel we will make intensive use of the Fenchel conjugate of a function on the Hilbert space. Recall that, given a function $g: X \to \R$, the Fenchel conjugate of $g$ is defined by
$$
g^*(x) = \sup_{z\in X} \lbrace \langle x,z \rangle-g(z) \rbrace, \quad x\in X,
$$
where $g^*$ may take the value $+\infty$ at some $x.$ We next gather some elementary properties of this operator which we will need later on. A detailed exposition can be found in \cite[Chapter 2, Section 3]{BorweinVanderwerffbook} or \cite[Chapter 2, Section 3]{Zalinescubook} for instance.

\begin{proposition}\label{elementarypropertiesconjugate} We have:
\item[$(i)$] $(ag)^*= ag^*( \frac{\cdot}{a})$ and $\left( ag( \frac{\cdot}{a}) \right)^*= a g^*$ for $a>0.$ 
\item[$(ii)$]If $\rho : \R \to \R$ is even, then $\left( \rho \circ \| \cdot \| \right)^* = \rho^* ( \| \cdot \| ) .$ 
\end{proposition}

Abusing of terminology, we will consider the Fenchel conjugate of nonnegative functions only defined on $[0,+\infty),$ say $\delta: [0,+\infty) \to [0,+\infty)$. In order to avoid problems, we will assume that all the functions involved are extended to all of $\R$ by setting $\delta(t)= \delta(-t)$ for $t<0$. Hence $\delta$ will be an even function on $\R$ and therefore
$$
\delta^*(t)= \sup_{ s\in \R} \lbrace ts-\delta(s) \rbrace =  \sup_{ s \geq 0} \lbrace ts-\delta(s) \rbrace ,\quad \text{for} \quad t \geq 0.
$$

\begin{proposition}\label{zalinescupropertiesomega}{\em [See \cite[Lemma 3.7.1, pg. 227]{Zalinescubook}.]}
We have that $\varphi^*(t)= \int_0^t \omega^{-1}(s) ds$ for all $t\geq 0$ and $\varphi(t)+\varphi^*(s)=ts$ if and only if $s= \omega(t).$ 
\end{proposition}

\begin{definition}
A function $f : X \to \R$ is said to be uniformly convex, with modulus of convexity $\delta$ (being $\delta : [0,+ \infty) \to [0,+ \infty)$ a nondecreasing function with $\delta(0)=0$) provided that
$$
\lambda f(x)+(1-\lambda) f(y) \geq f( \lambda x +(1-\lambda)y) + \lambda(1-\lambda) \delta( \| x-y\|)
$$
for all $\lambda \in [0,1]$ and $x,y\in X.$ 
\end{definition}

\begin{theorem}\label{vladimirovuniformlyconvexintegral}{\em [See \cite[Theorem 3]{Vladimirov}.]} Let $X$ be a Hilbert space. If $\rho: [0,+\infty) \to [0,+\infty)$ is an increasing function with $\rho( ct) \geq c \rho(t)$ for all $c \geq 1$ and $t \geq 0,$ then the function $\Phi : X \to \R$ defined by $\Phi(x) = \int_0^{\| x\|} \rho(t)dt, \: x\in X,$ is uniformly convex, with modulus of convexity $\delta(t)=\int_0^t \rho( s/2 ) ds,\: t \geq 0.$ 
\end{theorem}

For a mapping $G : E \to X$, where $E$ is a subset of $X$, we will denote 
$$
M_\omega(G)= \sup_{x \neq y ,\: x,y\in E} \frac{\| G(x)-G(y)\|}{\omega(\| x-y \|)}.
$$

\begin{proposition}\label{equivalenceomegadifferentiability}
Let $X$ be a Banach space. If $f:X\to\R$ is a continuous convex function and 
$$
f(x+h)+f(x-h)-2 f(x) \leq C \varphi( 2\| h \|), \quad \text{for all} \quad x,h \in X,
$$
then $f$ is of class $C^{1,\omega}(X)$ and $ \| D f(x)-D f(y) \| \leq 4 C \omega\left( 2 \|x-y\| \right) $ for all $x,y\in X.$  
\end{proposition}
\begin{proof}
The inequality of the assumption together with the continuity of $f$ proves the existence of $D f.$ Consider $x,y, h \in X$ with $\| h \| = \| x-y\|.$ Using repeatedly the convexity of $f$ and then the assumption, we get
\begin{align*}
( D f(x)- D f(y))( h ) &  \leq  f(x+h)-f(x)-Df(y)(h) \\
&  \leq f(x+h)-f(x)+f(x)-f(y)-Df(y)(x-y)-Df(y)(h) \\
&  \leq f(x+h)-f(y)-Df(y)(x+h-y) \\
&  \leq f(x+h)-f(y)-f(2y-x-h)-f(y) \\
& \leq f( y + (x+h-y) ) + f( y-(x+h-y) ) -2 f(y) \\
& \leq C  \varphi\left( 2 \| x+h-y\| \right) \leq C \varphi \left( 4 \|x-y\| \right).
\end{align*}
Thus 
$$
\| D f(x)-D f(y)\| \leq  4C \:  \frac{\varphi(4\|x-y\|)}{4\|x-y\|}.
$$
Note that, by concavity of $\omega,$ it follows that
$$
\frac{\varphi(t)}{t} = \int_0^1 \omega(tu) du \leq \omega \left( \frac{t}{2} \right) \quad t\geq 0.
$$
Therefore $ \| D f(x)-D f(y) \| \leq 4 C \omega\left( 2 \|x-y\| \right) .$
\end{proof}

\begin{lemma}\label{constantnormomega}
Let $(X, \|\cdot\|)$ be a Hilbert space, and $\varphi$ be defined by \eqref{definition of varphi}. Then the function $\psi(x)= \varphi(\| x\|), \: x\in X$, satisfies the following inequality
$$
\psi(x+h)+\psi(x-h)-2 \psi(x) \leq \psi(2h) \quad \text{for all} \quad x,h \in X.
$$
Also, $\psi $ is of class $C^{1,\omega}(X)$ with $ \| \nabla \psi(x)-\nabla \psi(y) \| \leq 4 \omega( 2 \|x-y\| )$ for all $x,y \in X.$ 
\end{lemma}
\begin{proof}
By combining the fact that $(\rho \circ \| \cdot \| )^*= \rho^*( \| \cdot \|)$ for any even $\rho: [0,+\infty) \to [0,+\infty)$ (see Proposition \ref{elementarypropertiesconjugate} and the subsequent comment) with Proposition \ref{zalinescupropertiesomega}, we obtain that $\psi^*(x)= \int_0^{\|x\|} \omega^{-1}(s)ds, \: x\in X,$ where $\omega^{-1}$ is a convex function. Thus, we can apply Theorem \ref{vladimirovuniformlyconvexintegral} with $\rho=\omega^{-1}$ and $\Phi= \psi^*$ to deduce that
$$
\lambda \psi^*(x)+(1-\lambda) \psi^*(y) \geq \psi^*( \lambda x +(1-\lambda)y) + \lambda(1-\lambda) \delta( \| x-y\|),
$$
for all $x,y\in X, \:\lambda \in [0,1],$ where $\delta(t)= \int_0^t \omega^{-1}\left( \frac{s}{2} \right) ds, \: t \geq 0.$ Then it is clear that
\begin{align*}
\delta_{\psi^*}(\varepsilon):& = \inf \left\lbrace \frac{1}{2} \psi^*(x)+\frac{1}{2}\psi^*(y)-\psi^* \left( \frac{x+y}{2} \right) \: : \|x-y\| \geq \varepsilon, \: x,y\in X \right\rbrace \\
& \geq \inf \left\lbrace \frac{1}{4} \delta( \|x-y\| ) \: : \|x-y \| \geq \varepsilon, \: x,y\in X \right\rbrace \geq \frac{1}{4} \delta( \varepsilon ) 
\end{align*}
for all $\varepsilon \geq 0.$ Let us denote
$$
\rho_{\psi}(t) := \sup \left\lbrace \frac{1}{2}\psi(x+ty)+\frac{1}{2}\psi(x-ty)-\psi(x) \: : x,y \in X, \: \|y\|=1 \right\rbrace
$$
for all $t\geq 0.$
Since $\psi$ is continuous and convex on $X,$ we can use \cite[Theorem 5.4.1(a), pg. 252]{BorweinVanderwerffbook} to deduce 
$$
\rho_{\psi}(t) = \sup \left\lbrace t \tfrac{\varepsilon}{2} - \delta_{\psi^*}(\varepsilon ) \: : \varepsilon \geq 0 \right\rbrace, \quad t \geq 0.
$$
Applying the preceding estimation to $\delta_{\psi^*}$ we see that
$$
\rho_{\psi}(t) \leq \tfrac{1}{2} \sup \left\lbrace t \varepsilon - \tfrac{1}{2} \delta(\varepsilon) \: : \varepsilon \geq 0 \right\rbrace = \tfrac{1}{2}\left( \tfrac{1}{2} \delta \right)^*(t), \quad t \geq 0.
$$
By definition of $\delta$ it is clear that $\frac{1}{2} \delta(t)= \int_0^{t/2} \omega^{-1}(s)ds.$ Using Proposition \ref{elementarypropertiesconjugate} together with Proposition \ref{zalinescupropertiesomega} we have that $ \left( \tfrac{1}{2} \delta \right)^*(t)= \int_0^{2t} \omega(s)ds, \: t \geq 0.$ Then it follows
$$
\rho_{\psi}(t) \leq \frac{1}{2} \int_0^{2t} \omega(s)ds, \quad t \geq 0
$$
and therefore
$$
\psi(x+ty)+\psi(x-ty)-2\psi(x) \leq \int_0^{2t} \omega(s) ds, \quad \textrm{ for all } t\geq 0,\: x,y\in X,\: \textrm{ with } \|y\|=1,
$$
which is equivalent to the desired inequality. The second part follows from Proposition \ref{equivalenceomegadifferentiability}.

\end{proof}

\begin{definition}\label{definitionconditioncw1omega}
Given an arbitrary subset $E$ of a Hilbert space $X$, and a $1$-jet $f: E \to \R, \: G : E \to X,$ we will say that $(f,G)$ satisfies condition $(CW^{1,\omega})$ on $E$ with constant $M>0,$ provided that
$$
f(x) \geq f(y)+ \langle G(y), x-y \rangle + M \varphi^*\left(\frac{1}{M} \| G(x)-G(y)\| \right) \quad \text{for all} \quad x,y\in E. \eqno (CW^{1, \omega})
$$
\end{definition}

\begin{remark}\label{remarkrewritingcw1omega} We have:
\item[$(i)$]
If $(f,G)$ satisfies $(CW^{1,\omega})$ on $E$ with constant $M,$ then
$$
\| G(x)-G(y) \| \leq 2M \omega \left( \frac{\|x-y\|}{2} \right) \quad x,y \in E.
$$
In particular $M_\omega(G) \leq 2M.$
\item[$(ii)$] The inequality defining condition $(CW^{1,\omega})$ can be rewritten as
$$
f(x) \geq f(y)+ \langle G(y), x-y \rangle + (M \varphi)^*\left( \| G(x)-G(y)\| \right) \quad \text{for all} \quad x,y\in E,
$$
\end{remark}

\begin{proof}
\item[$(i)$] We fix $x,y\in E$ and set $t = \frac{1}{M} \| G(x)-G(y)\|.$ We have that
$$
M \varphi^*\left( t \right) = \| G(x)-G(y)\| \frac{\varphi^*(t)}{t}.
$$
Using first Proposition \ref{zalinescupropertiesomega} and then Jensen's inequality (recall that $\omega^{-1}$ is a convex function) we obtain
$$
\frac{\varphi^*(t)}{t} = \int_0^1 \omega^{-1}(tu) du \geq \omega^{-1}\left( \frac{t}{2} \right) = \omega^{-1} \left( \frac{1}{2M} \| G(x)-G(y)\| \right)
$$
and then
$$
M \varphi^*\left(\frac{1}{M} \| G(x)-G(y)\| \right) \geq \| G(x)-G(y)\| \omega^{-1} \left( \frac{1}{2M} \| G(x)-G(y)\| \right).
$$
Now, using the inequality defining the condition $(CW^{1,\omega})$ we have
\begin{align*}
f(x) \geq f(y)+ \langle G(y), x-y \rangle + M \varphi^*\left(\frac{1}{M} \| G(x)-G(y)\| \right) \\
f(y) \geq f(x)+ \langle G(x), y-x \rangle + M \varphi^*\left(\frac{1}{M} \| G(x)-G(y)\| \right)
\end{align*}
hence
\begin{align*}
\langle  G(x)-G(y), & x-y \rangle  \geq 2M \varphi^*\left(\frac{1}{M} \| G(x)-G(y)\| \right) \\
& \geq \| G(x)-G(y)\| \omega^{-1} \left( \frac{1}{2M} \| G(x)-G(y)\| \right).
\end{align*}
We conclude that 
$$
\| G(x)-G(y) \| \leq 2M \omega \left( \frac{\|x-y\|}{2} \right).
$$
\item[$(ii)$] This follows from elementary properties of the conjugate of a function; see Proposition \ref{elementarypropertiesconjugate}.
\end{proof}

\begin{remark}
In \cite{AM}, one can find an alternative formulation of the condition $(CW^{1,\omega})$ for a $1$-jet $(f,G)$ on $E$, namely:
\begin{equation}\label{oldcw1omega}
f(x) \geq f(y)+ \langle G(y), x-y \rangle + \| G(x)-G(y)\| \omega^{-1}\left( \frac{1}{2M} \|G(x)-G(y)\| \right)
\end{equation}
for all $x,y\in E.$ If we denote the above condition by $\widetilde{(CW^{1,\omega})},$ we have that $\widetilde{(CW^{1,\omega})}$ and $(CW^{1,\omega})$ are actually equivalent.
\end{remark}
\begin{proof}
Since $\omega^{-1}$ is convex, we have that 
\begin{equation}\label{comparingomega-11}
\varphi^*(t)=\int_0^t \omega^{-1}(s) ds \geq t \omega^{-1}\left( t/2 \right) \quad \text{for all} \quad t \geq 0.
\end{equation}
On the other hand, because $\omega^{-1}$ is increasing we easily obtain 
\begin{equation}\label{comparingomega-12}
\varphi^*(t) \leq t \omega^{-1}(t) \quad \text{for all} \quad t\geq 0.
\end{equation}
Taking first $t=\tfrac{1}{M}\|G(x)-G(y)\|$ in \eqref{comparingomega-11} and then $t=\tfrac{1}{2M}\|G(x)-G(y)\|$ in \eqref{comparingomega-12} and also bearing in mind Proposition \ref{elementarypropertiesconjugate} $(i)$ we easily obtain
\begin{align*}
(M \varphi)^* \left( \| G(x)-G(y)\| \right) & \geq \|G(x)-G(y)\| \omega^{-1}\left( \frac{1}{2M}\| G(x)-G(y)\| \right) \\
& \geq (2 M\varphi)^*(\|G(x)-G(y)\|).
\end{align*}
By comparing condition $(CW^{1,\omega})$ (Definition \ref{definitionconditioncw1omega}) with $\widetilde{(CW^{1,\omega})}$ (inequality \eqref{oldcw1omega}) we then see that both conditions are equivalent.
\end{proof}

Let us now see that $(CW^{1, \omega})$ is a necessary condition for $C^{1, \omega}$ convex extension of $1$-jets.
\begin{proposition}\label{necessityconditioncw1omega}
Let $f\in C^{1,\omega}(X)$ be convex, and assume that $f$ is not affine. Then the $1$-jet $(f,\nabla f)$ satisfies the condition $(CW^{1,\omega})$ with constant $M>0$ on $E=X$, where 
$$
M=\sup_{x, y\in X, \, x\neq y}\frac{\|\nabla f(x)-\nabla f(y)\|}{\omega(\|x-y\|)}.
$$
\end{proposition}
\noindent On the other hand, if $f$ is affine, it is obvious that $(f, \nabla f)$ satisfies $(CW^{1,1})$ on every $E\subset X$, for every $M>0$. 
\begin{proof}
Suppose that there exist different points $x, y\in X$ such that
$$
f(x)-f(y)-\langle \nabla f(y), x-y\rangle < M \varphi^*\left(\frac{1}{M} \| \nabla f(x)-\nabla f(y)\| \right),
$$
and we will get a contradiction.

\noindent {\bf Case 1.} Assume further that $M=1$, $f(y)=0$, and $\nabla f(y)=0$.
By convexity this implies $f(x)\geq 0$.
Then we have 
$$
0\leq f(x)<\varphi^*\left( \| \nabla f(x)\| \right).
$$
Set 
$$
v=-\frac{1}{\|\nabla f(x)\|}\nabla f(x),
$$
and define 
$$
h(t)=f(x+tv)
$$
for every $t\in\R$. We have $h(0)=f(x)$, $h'(0)=-\| \nabla f(x)\|$, and $h'(t)=\langle \nabla f(x+tv), v\rangle.$ This implies that 
$$
\big | h(t)-f(x)+\| \nabla f (x)\|t \big |\leq \int_0^t \omega(s) ds = \varphi(t)
$$
for every $t\in\R^{+}$, hence also that 
$$
h(t)\leq -\|\nabla f(x) \| t+f(x)+\varphi(t) \textrm{ for all } t\in \R^{+}.
$$
By using the assumption on $f(x)$ and Proposition \ref{zalinescupropertiesomega} we have
\begin{align*}
& f\left( x+ \omega^{-1}(\| \nabla f(x)\| ) v\right)\\
& <  \varphi^*( \| \nabla f(x)\| ) - \| \nabla f(x) \| \omega^{-1}( \| \nabla f(x) \| ) + \varphi \left( \omega^{-1}( \| \nabla f(x) \| ) \right)=0,
\end{align*}
which is in contradiction with the assumptions that $f$ is convex, $f(y)=0$, and $\nabla f(y)=0$. This shows that
$$
f(x)\geq \varphi^*\left( \| \nabla f(x)\| \right).
$$

\noindent {\bf Case 2.} Assume only that $M=1$. Define
$$
g(z)=f(z)-f(y)-\langle \nabla f(y), z-y\rangle
$$
for every $z\in X$. Then $g(y)=0$ and $\nabla g(y)=0$. By Case 1, we get
$$
g(x)\geq \varphi^*\left( \|\nabla g(x)\| \right),
$$
and since $\nabla g(x)=\nabla f(x)-\nabla f(y)$ the Proposition is thus proved in the case when $M=1$.

\noindent {\bf Case 3.} In the general case, we may assume $M>0$ (the result is trivial for $M=0$). Consider $g=\frac{1}{M}f$, which satisfies the assumption of Case 2. Therefore
$$
g(x)-g(y)-\langle \nabla g(y), x-y\rangle \geq \varphi^*\left( \|\nabla g(x)-\nabla g(y)\| \right),
$$
which is equivalent to the desired inequality.
\end{proof}

Let us now present the main result of this section.
\begin{theorem}\label{theoremformulac1omegaconvex}
Given a $1$-jet $(f,G)$ defined on $E$ satisfying the property $(CW^{1,\omega})$ with constant $M$ on $E,$ the formula
$$
F=\textrm{conv}(g), \quad g(x) = \inf_{y \in E} \lbrace f(y)+\langle G(y), x-y \rangle + M \varphi \left( \|x-y\| \right) \rbrace, \quad x\in X,
$$
defines a $C^{1,\omega}$ convex function with $F_{|_E}=f$ and $(\nabla F)_{|_E}  =G$, and 
$$
\| \nabla F(x)-\nabla F(y) \| \leq 4M \omega\left( 2 \| x-y\| \right) \quad \text{for all} \quad x,y \in X.
$$
In particular, $M_\omega( \nabla F) \leq 8 M.$
\end{theorem}

For the proof we will use the following auxiliary results.
\begin{proposition}[Generalized Young's inequality for the Fenchel conjugate]\label{youngfenchelinequality}
Let $\rho:\R \to \R $ be a convex function. Then
$$
ab \leq \rho(a) + \rho^*(b) \quad \text{for all} \quad a,b>0.
$$
\end{proposition}

\begin{lemma}\label{minimalsmallerthaninfalpha}
We have 
$$
f(z)+ \langle G(z), x-z \rangle \leq f(y)+\langle G(y), x-y \rangle + M \varphi( \|x-y\|) 
$$
for every $y,z\in E, \: x\in X.$ 
\end{lemma}
\begin{proof}
Given $y,z\in E, \: x\in X,$ condition $(CW^{1,\omega})$ with constant $M$ (together with Remark \ref{remarkrewritingcw1omega} $(ii)$) leads us to
\begin{align*}
 & f(y)  +\langle G(y),  x-y \rangle  + M \varphi( \| x-y\| )   \\
& \geq f(z)+ \langle G(z),x-z \rangle + (M \varphi)^*(\| G(y)-G(z)\|) \\
 & \quad + \langle G(z)-G(y), y-x \rangle + M \varphi( \|x-y\|) \\
 & \geq f(z)+ \langle G(z),x-z \rangle -ab+ M \varphi(a) +(M\varphi)^*(b),
\end{align*}
where $a=\|y-x\|$ and $b= \| G(z)-G(y)||.$ Applying Proposition \ref{youngfenchelinequality} we obtain that the last term is greater than or equal to $f(z)+\langle G(z), x-z \rangle.$ 
\end{proof}

The previous Lemma shows that $m \leq g$, where $g$ is defined as in Theorem \ref{theoremformulac1omegaconvex}, and 
$$
m(x) := \sup_{z \in E} \lbrace f(z)+ \langle G(z), x-z \rangle \rbrace, \quad x\in X.
$$
By definition of $g$ and $m$ it is then obvious that $f \leq m \leq g \leq f$ on $E.$ Thus $g = f$ on $E.$ 

\begin{lemma}
We have
$$
g(x+h)+g(x-h)-2g(x) \leq M \varphi(\| 2 h \|) \quad \text{for all} \quad x,h \in X.
$$
\end{lemma}
\begin{proof}
Given $x,h \in X$ and $\varepsilon>0,$ by definition of $g,$ we can pick $y\in E$ with
$$
g(x) \geq f(y)+ \langle G(y), x-y \rangle + M \varphi( \| x-y\|) - \varepsilon.
$$
We then have
\begin{align*}
g(x+h)& +g(x-h)-2g(x)  \leq f(y)+\langle G(y), x+h-y \rangle + M \varphi( \| x+h-y\|)\\
& \quad + f(y)+\langle G(y), x-h-y \rangle  + M \varphi( \| x-h-y\|) \\
& \quad -2 \left( f(y)+ \langle G(y), x-y \rangle +  M \varphi( \| x-y\|) \right) + 2 \varepsilon \\
& = M \left( \varphi(\|x+h-y\|) + \varphi(\|x-h-y\|)- 2 \varphi(\| x-y\|)\right) + 2 \varepsilon \\
& \leq M \varphi(2\| h \|) + 2 \varepsilon,
\end{align*}
where the last inequality follows from Lemma \ref{constantnormomega}.
\end{proof}

Now, if we define $F = \textrm{conv}(g)$, with the same proof as that of Theorem \ref{convexenvelopeconstant}, we get that
$$
F(x+h)+F(x-h)-2 F(x) \leq M \varphi(\|2 h \|)  \quad \text{for all} \quad x,h \in X.
$$
Because $F$ is convex, by virtue of Proposition \ref{equivalenceomegadifferentiability}, we have that $F \in C^{1,\omega}(X)$ with 
$$
\| \nabla F(x)-\nabla F(y) \|\leq 4M \omega( 2 \|x-y\| ) \quad \text{for all} \quad x,y \in X.
$$
Finally, the same argument involving the function $m$ as that at the end of Section \ref{sectionc11convex} shows that $F=f$ and $\nabla F = G$ on $E.$

\section{$C^{1,\alpha}$ extensions of convex jets in superreflexive Banach spaces}

Throughout this section, and unless otherwise stated, $X$ will denote a superreflexive Banach space, $\| \cdot \|$ an equivalent norm on $X$ and $\| \cdot \|_*$ the dual norm of $\| \cdot \|$ on $X^*$. By Pisier's results (see \cite[Theorem 3.1]{Pisier}), we may assume that the norm $\| \cdot \|$ is uniformly smooth with modulus of smoothness of power type $p=1+\alpha$ for some $0<\alpha \leq 1$. Hence, there exists a constant $C \geq 2$, depending only on this norm, such that
\begin{equation}\label{modulussmoothnesssuperreflexive}
 \| x+h \|^{1+\alpha} +  \| x-h\|^{1+\alpha}-2 \|x\|^{1+\alpha} \leq C \| h\|^{1+\alpha} \quad \text{for all} \quad x,h\in X.
\end{equation}
For a mapping $G : E \to X^*,$ where $E$ is a subset of $X,$ we will denote 
$$
M_\alpha(G)= \sup_{x \neq y ,\: x,y\in E} \frac{\| G(x)-G(y)\|_*}{\| x-y \|^\alpha}.
$$

By a $1$-jet defined on $E$ we mean a pair of functions $(f,G),$ where $f: E \to \R$ and $G: E \to X^*$. 

\begin{definition}
Given an arbitrary subset $E$ of $X,$ and a $1$-jet $f: E \to \R, \: G : E \to X^*,$ we will say that $(f,G)$ satisfies the condition $(CW^{1,\alpha})$ on $E$ with constant $M>0,$ provided that
$$
f(x) \geq f(y)+ G(y)( x-y)+ \frac{\alpha}{(1+\alpha)M^{1/\alpha}} \| G(x)-G(y)\|_*^{1+ \frac{1}{\alpha}} \quad \text{for all} \quad x,y\in E.
$$
\end{definition}

\begin{remark}
If $(f,G)$ satisfies $(CW^{1,\alpha})$ on $E$ with constant $M>0,$ then $M_\alpha(G) \leq \left( \frac{1+\alpha}{2 \alpha} \right)^\alpha M.$
\end{remark}
\begin{proof}
Using inequality $(CW^{1,\alpha})$ we obtain for all $x,y\in E$
\begin{align*}
f(x) \geq f(y)+ G(y)( x-y) + \frac{\alpha}{(1+\alpha)M^{1/\alpha}} \| G(x)-G(y)\|_*^{1+ \frac{1}{\alpha}} , \\
f(y) \geq f(x)+ G(x)( y-x) + \frac{\alpha}{(1+\alpha)M^{1/\alpha}} \| G(x)-G(y)\|_*^{1+ \frac{1}{\alpha}} .
\end{align*}
By summing up both inequalities we easily get
$$
\|G(x)-G(y)\|_*\, \| x-y \| \geq  \left( G(x)-G(y)\right) (x-y) \geq \frac{2\alpha}{(1+\alpha)M^{1/\alpha}} \| G(x)-G(y)\|_*^{1+ \frac{1}{\alpha}}
$$
which immediately implies the desired estimate.
\end{proof}

\begin{proposition}\label{necessityconditioncw1alpha}
Let $X$ be a Banach space, let $f\in C^{1,\alpha}(X)$ be convex with $M_\alpha(Df)\leq M,$ and assume that $f$ is not affine. Then $(f,Df)$ satisfies the condition $(CW^{1,\alpha})$ on $X$ with constant $M.$ 
\end{proposition}
\noindent On the other hand, if $f$ is affine and continuous, it is obvious that $(f, D f)$ satisfies $(CW^{1,\alpha})$ on every $E\subset X$, for every $M>0$. 
\begin{proof}
Suppose that there exist different points $x, y\in X$ such that
$$
f(x)-f(y)- D f(y) ( x-y) <\frac{\alpha}{(1+\alpha) M^{\frac{1}{\alpha}}}\| D f(x)-D f(y)\|_*^{1+ \frac{1}{\alpha}},
$$
and we will get a contradiction.

\noindent {\bf Case 1.} Assume further that $M=1$, $f(y)=0$, and $D f(y)=0$.
By convexity this implies $f(x)\geq 0$.
Then we have 
\begin{equation}\label{assumptionfxcontradiction}
0\leq f(x)\leq \frac{\alpha}{1+\alpha}\|D f(x)\|_*^{1+ \frac{1}{\alpha}}-r \quad \text{for some} \quad r>0.
\end{equation}
Let us fix $0<\varepsilon \leq \frac{r}{2}\big \| Df(x)\big \|_*^{-\left( 1+\frac{1}{\alpha} \right)}$ and pick $v_\varepsilon \in X$ with $\|v_\varepsilon\|=1$ and 
\begin{equation}\label{approximationnormvvarepsilon}
Df(x)(v_\varepsilon) \leq (\varepsilon-1) \| Df(x)\|_*.
\end{equation}
We define $\varphi(t)=f(x+tv_\varepsilon)
$
for every $t\in\R.$ We have $\varphi(0)=f(x)$, $\varphi'(0)=Df(x)(v_\varepsilon)$, and $\varphi'(t)=D f(x+tv_\varepsilon)(v_\varepsilon).$ This implies that 
$$
|\varphi(t)-\varphi(0)-\varphi'(0)t|\leq \int_0^t s^\alpha ds = \frac{t^{1+\alpha}}{1+\alpha}
$$
for every $t\in\R^{+}$, hence also that 
$$
\varphi(t)\leq f(x) +t  Df(x)(v_\varepsilon)+\frac{t^{1+\alpha}}{1+\alpha} \textrm{ for all } t\in \R^{+}.
$$
Using first \eqref{assumptionfxcontradiction} and then \eqref{approximationnormvvarepsilon} we have
\begin{align*}
& f\left( x+ \| Df(x)\|_*^{1/\alpha} v_\varepsilon\right) =\varphi\left( \| Df(x)\|_*^{1/\alpha} \right) \\
&  \leq \frac{\alpha}{1+\alpha}\| D f(x)\|_*^{1+\frac{1}{\alpha}}-r +  \| Df(x)\|_*^{1/\alpha} Df(x)(v_\varepsilon)+\frac{1}{1+ \alpha}\|Df(x)\|_*^{1+\frac{1}{\alpha}} \\
& \leq -r + \varepsilon \|Df(x)\|_*^{1+\frac{1}{\alpha}} \leq -\frac{r}{2}<0,
\end{align*}
which is in contradiction with the assumptions that $f$ is convex, $f(y)=0$, and $D f(y)=0.$ This shows that
$$
f(x)\geq \frac{\alpha}{1+\alpha}\|D f(x)\|_*^{1+ \frac{1}{\alpha}}.
$$

\noindent {\bf Case 2.} Assume only that $M=1$. Define
$
g(z)=f(z)-f(y)-D f(y)(z-y)
$
for every $z\in X$. Then $g(y)=0$ and $D g(y)=0$. By Case 1, we get
$$
g(x)\geq \frac{\alpha}{1+\alpha}\|D g(x)\|_*^{1+ \frac{1}{\alpha}},
$$
and since $D g(x)=D f(x)-D f(y)$ the Proposition is thus proved in the case when $M=1$.

\noindent {\bf Case 3.} In the general case, we may assume $M>0$ (the result is trivial for $M=0$). Consider $\psi=\frac{1}{M}f$, which satisfies the assumption of Case 2. Therefore
$$
\psi(x)-\psi(y)-D \psi(y) ( x-y) \geq \frac{\alpha}{1+\alpha} \|D \psi(x)-D \psi(y)\|_*^{1+ \frac{1}{\alpha}},
$$
which is equivalent to the desired inequality.
\end{proof}

\begin{proposition}\label{equivalencealphadifferentiability}
If $f$ is a continuous convex function on $X$ and 
$$
f(x+h)+f(x-h)-2 f(x) \leq C \| h \|^{1+\alpha}, \quad \text{for all} \quad x,h \in X,
$$
then $f$ is of class $C^{1,\alpha}(X)$ and $M_\alpha(D f) \leq  2^{1+\alpha} C . $ 
\end{proposition}
\begin{proof}
Similar to the proof of Proposition \ref{equivalenceomegadifferentiability}.
\end{proof}

Our main result in this section is the following.

\begin{theorem}\label{theoremformulac1alphaconvex}
Given a $1$-jet $(f,G)$ defined on $E$ satisfying the property $(CW^{1,\alpha})$ with constant $M$ on $E,$ the formula
$$
F=\textrm{conv}(g), \quad g(x) = \inf_{y \in E} \lbrace f(y)+G(y)(x-y) + \tfrac{M}{1+\alpha} \|x-y\|^{1+\alpha} \rbrace, \quad x\in X,
$$
defines a $C^{1,\alpha}$ convex function with $F_{|_E}=f$, $(D F)_{|_E}=G$, and $$M_\alpha(D F) \leq \frac{2^{1+\alpha} C }{1+\alpha}M,$$ where $C$ is the constant of \eqref{modulussmoothnesssuperreflexive}.
\end{theorem}
\begin{proof}
The general scheme of the proof is similar to that of Theorem \ref{theoremformulaC11convex}. We will need to use the following auxiliary results.
\begin{proposition}[Young's inequality]\label{youngsinequality}
Let $1<p,q < \infty$ with $\tfrac{1}{p}+ \tfrac{1}{q}=1.$ Then
$$
ab \leq \varepsilon a^p + \frac{b^q}{q( \varepsilon p)^{q/p}} \quad a,b,\varepsilon >0.
$$
\end{proposition}

\begin{lemma}\label{minimalsmallerthaninfalpha}
We have 
$$
f(z)+ G(z)( x-z ) \leq f(y)+ G(y)(x-y) + \tfrac{M}{1+\alpha} \|x-y\|^{1+\alpha} 
$$
for every $y,z\in E, \: x\in X.$ 
\end{lemma}
\begin{proof}
Given $y,z\in E, \: x\in X,$ condition $(CW^{1,\alpha})$ with constant $M$ implies
\begin{align*}
  f(y)&  + G(y)(  x-y)  + \tfrac{M}{1+\alpha} \|x-y\|^{1+\alpha}   \\
 & \geq f(z)+  G(z)( y-z)  + \tfrac{\alpha}{(1+\alpha)M^{1/\alpha}}\| G(y)-G(z)\|_*^{1+ \frac{1}{\alpha}} \\
 & \quad +  G(y)( x-y) + \tfrac{M}{1+\alpha} \|x-y\|^{1+\alpha} \\
 & = f(z)+  G(z)(x-z)  + \tfrac{\alpha}{(1+\alpha)M^{1/\alpha}}\| G(y)-G(z)\|_*^{1+ \frac{1}{\alpha}} \\
 & \quad +  (G(z)-G(y))( y-x)  + \tfrac{M}{1+\alpha} \|x-y\|^{1+\alpha} \\
 & \geq f(z)+  G(z)(x-z)  -ab+ \tfrac{M}{1+\alpha}a^{1+\alpha} + \tfrac{\alpha}{(1+\alpha)M^{1/\alpha}}b^{1+\frac{1}{\alpha}},
\end{align*}
where $a=\|y-x\|$ and $b= \| G(z)-G(y)\|_*$. By applying Proposition \ref{youngsinequality} with
$$
p=1+\alpha, \quad q= 1+\tfrac{1}{\alpha} \quad \varepsilon = \tfrac{M}{1+\alpha},
$$
we obtain that 
$$
 -ab+ \tfrac{M}{1+\alpha}a^{1+\alpha} + \tfrac{\alpha}{(1+\alpha)M^{1/\alpha}}b^{1+\frac{1}{\alpha}} \geq 0.
 $$
This proves the Lemma.
\end{proof}

The preceding lemma shows that $m \leq g$, where $g$ is defined as in Theorem \ref{theoremformulac1alphaconvex}, and 
$$
m(x) := \sup_{z \in E} \lbrace f(z)+ G(z) (x-z) \rbrace, \quad x\in X.
$$
Then, using the definition of $g$ and $m$, we also have that $f \leq m \leq g \leq f$ on $E$. Thus $g = f$ on $E.$

\begin{lemma}
We have
$$
g(x+h)+g(x-h)-2g(x) \leq \frac{C M}{1+\alpha} \| h \|^{1+\alpha} \quad \text{for all} \quad x,h \in X,
$$
where $C$ is as in \eqref{modulussmoothnesssuperreflexive}.
\end{lemma}
\begin{proof}
Given $x,h \in X$ and $\varepsilon>0,$ by definition of $g,$ we can pick $y\in E$ with
$$
g(x) \geq f(y)+ G(y) (x-y) + \tfrac{M}{1+\alpha} \| x-y\|^{1+\alpha} - \varepsilon.
$$
We then have
\begin{align*}
g(x+h)& +g(x-h)-2g(x)  \leq f(y)+ G(y) (x+h-y)  + \tfrac{M}{1+\alpha} \| x+h-y\|^{1+\alpha} \\
& \quad + f(y)+G(y) (x-h-y)  + \tfrac{M}{1+\alpha} \| x-h-y\|^{1+\alpha} \\
& \quad -2 \left( f(y)+  G(y) ( x-y) +  \tfrac{M}{1+\alpha} \| x-y\|^{1+\alpha} \right) + 2 \varepsilon \\
& = \tfrac{M}{1+\alpha} \left( \|x+h-y\|^{1+\alpha} + \|x-h-y\|^{1+\alpha}- 2 \| x-y\|^{1+\alpha} \right) + 2 \varepsilon \\
& \leq \frac{C M}{1+\alpha} \| h \|^{1+\alpha} + 2 \varepsilon,
\end{align*}
where the last inequality follows from inequality \eqref{modulussmoothnesssuperreflexive}.
\end{proof}

Then, by defining $F = \textrm{conv}(g)$, and with the same proof as that of Theorem \ref{convexenvelopeconstant}, we deduce that
$$
F(x+h)+F(x-h)-2 F(x) \leq \frac{C M}{1+\alpha} \| h \|^{1+\alpha}  \quad \text{for all} \quad x,h \in X.
$$
Because $F$ is convex and continuous, by virtue of Proposition \ref{equivalencealphadifferentiability}, we have that $F \in C^{1,\alpha}(X)$ with 
$$
M_\alpha( D F) \leq \frac{2^{1+\alpha} C }{1+\alpha} M.
$$
Finally, the same argument involving the function $m$ as that at the end of Section \ref{sectionc11convex} shows that $F=f$ and $D F = G$ on $E$.
\end{proof}

\section{Final comments}

Let us finish this paper with some comments and an example which show that we cannot expect the above results to hold true for a general Banach space $X$, unless $X$ is superreflexive. 

On the one hand, observe that a necessary condition for the validity of a Whitney extension theorem of class $C^{1, \omega}(X)$ for a Banach space $X$ is that there is a smooth bump function whose derivative is $\omega$-continuous on $X$. Indeed, let $C=\{x\in X : \|x\|\geq 1\} \cup\{0\}$, and define $f:C\to\R$ and $G:C\to X^{*}$ by
$$
f(x)=0 \textrm{ if } \|x\|\geq 1, \, f(0)=1, \textrm{ and } G(x)=0 \textrm{ for all } x\in C.
$$
It is trivial to check that the jet $(f,G)$ satisfies the assumptions $(\widetilde{W^{1,1}})$ of the Whitney extension theorem. If a Whitney-type extension theorem were true for $X$, then there would exist a $C^{1, \omega}$ function $F:X\to\R$ such that $F(x)=0$ for $\|x\|\geq 1$ and $F(0)=1$. Then according to \cite[Theorem V.3.2]{DGZ} the space $X$ would be superreflexive. 

It is unkown whether for every superreflexive Banach space $X$ (other than a Hilbert space) a Whitney-type extension theorem for the class $C^{1, \omega}$ holds true at least for some modulus $\omega$. It is also unknown whether a Whitney-type extension theorem holds true for every class $C^{1, \omega}(X)$ if $X$ is a Hilbert space and $\omega$ is not linear. However the results of this paper provide some answers to analogous questions for the classes $C^{1, \omega}_{\textrm{conv}}(X)$.

On the other hand, one could ask whether superreflexivity of $X$ is necessary in order to obtain Whitney-type extension theorems for the classes $C^{1, \omega}_{\textrm{conv}}(X)$, and wonder whether Banach spaces like $c_0$, with sufficiently many differentiable functions (and even with real-analytic equivalent norms), could admit such Whitney-type theorems. The following example answers this question in the negative.
\begin{example}
Let $X=c_0$ (the Banach space of all sequences of real numbers that converge to $0$, endowed with the sup norm).
Then for every modulus of continuity there are discrete sets $C\subset X$ and $1$-jets $(f,G)$ with $f:C\to\R$, $G:C\to X^{*}$ satisfying condition $(CW^{1, \omega})$ on $C$, and such that for no $F\in C^{1, \omega}_{\textrm{conv}}(X)$ do we have $F_{|_C}=f$ and $(\nabla F)_{|_C}=G$.
\end{example}
For simplicity we will only give the proof in the case that $\omega(t)=t^{\alpha}$, that is to say, for the classes $C^{1, \alpha}(X)$. In the general case the same proof works, with obvious changes.
Let $\{e_{j}\}_{j=1}^{\infty}$ be the canonical basis of $X$ (that is to say $e_1= (1, 0, 0, \ldots),$ $e_2=(0, 1, 0, \ldots)$, etc), and let $\{e_{j}^{*}\}_{j=1}^{\infty}$ be the associated coordinate functionals; thus we have that $\|e_j\|=1$, $e_{i}^{*}(e_j)=\delta_{ij}$, and $\|e_{j}^{*}\|_{*}=1$. 
Let
$$
C=\{\pm e_{j}: j\in\N\}\cup\{0\},
$$
and define $f:C\to\R$ and $G:C\to X^{*}$ by
$$
f(0)=0, f(\pm e_j)=\frac{1}{2} \textrm{ for all } j\in\N, \textrm{ and } G(0)=0, G(\pm e_j)=\pm e_{j}^{*} \textrm{ for all } j\in\N.
$$
It is easy to check that
$$
f(x)-f(y)-G(y)(x-y)\geq \frac{1}{2}= \frac{2^{1+\frac{1}{\alpha}}}{2^{2+\frac{1}{\alpha}}}\geq \frac{1}{2^{2+\frac{1}{\alpha}}}\|G(x)-G(y)\|_{*}^{1+\frac{1}{\alpha}} \textrm{ for all } x, y\in C, \: x \neq y,
$$
hence  $(f,G)$ satisfies property $(CW^{1, \alpha})$ on $C$. Assume now that there exists $F\in C^{1, \alpha}_{\textrm{conv}}(X)$ such that $(F, DF)$ extends the jet $(f,G)$.
If $\|x\|=1$ then, by taking $j\in\N$ such that $|x_j|=1$, we have, either with $y_j=e_j$ or with $y_j=-e_j$, that
$$
F(x)\geq F(y_j)+DF(y_j)(x-y_j)=\frac{1}{2}+|x_j|-1=\frac{1}{2},
$$
and by convexity it follows that $F(x)\geq 1/2$ for all $\|x\|\geq 1$, while $F(0)=0$. Then, by composing $F$ with a suitable real function, we may easily obtain a $C^{1, \alpha}$ function $\varphi:X\to [0,1]$ with $\varphi(0)=1$ and $\varphi(x)=0$ for all $\|x\|\geq 1$. But then again, using for instance \cite[Theorem V.3.2]{DGZ}, $X=c_0$ would be a superreflexive space, which is absurd. \qed

Observe also that the proof of Proposition \ref{necessityconditioncw1alpha} shows that $(CW^{1, \alpha})$ is a necessary condition for $C^{1, \alpha}_{\textrm{conv}}(X)$ extension. The above example shows that in the case that $X=c_0$ this condition is no longer sufficient, and therefore any characterization of the class of $1$-jets defined on subsets of $c_0$ which admit $C^{1, \alpha}_{\textrm{conv}}$ extensions to $c_0$ would have to involve some new conditions.

\section{Acknowledgements}
We are very grateful to Fedor Nazarov for pointing out to us that, even though the function $g$ in the statement of Theorem \ref{theoremformulaC11convex} is not differentiable in general, its convex envelope always is of class $C^{1,1}$. We also wish to thank the referee for several suggestions that improved the exposition.

\end{document}